\documentclass{amsart}
\usepackage[T1]{fontenc}   

\usepackage{tikz}
\usetikzlibrary{cd}
\usepackage[all,cmtip]{xy}
\usepackage{amsmath,amsthm,amssymb,mathrsfs}
\usepackage{mathtools,bbm,latexsym}
\usepackage{color,enumitem}
\usepackage{verbatim} 
\usepackage{hyperref}
\usepackage{booktabs}
\input diagxy

\allowdisplaybreaks  

\DeclareMathOperator{\sub}{sub}

\DeclareMathOperator{\colim}{colim}
 
\DeclareMathOperator{\with}{\&}

\DeclareMathOperator{\wayb}{{\rotatebox[origin=c]{-90}{$\twoheadrightarrow$}}}

\newcommand{\CC}{\mathcal{C}}

\newcommand{\CP}{\mathcal{P}}
\newcommand{\CPd}{\mathcal{P}^\dag}

\newcommand{\CI}{\mathcal{I}}
\newcommand{\CJ}{\mathcal{J}}

\newcommand{\lam}{\lambda}
\newcommand{\bv}{\sup}
\newcommand{\bw}{\inf}

\newcommand{\ra}{\rightarrow}

\newcommand{\la}{\leftarrow}

\newcommand{\sy}{{\sf y}} 
\newcommand{\syd}{{\sf y}^\dag}
 
\newcommand{\sV}{{\sf V}}

\newcommand{\oto}{{\to\hspace*{-3.1ex}{\circ}\hspace*{1.7ex}}}

\newtheorem{theorem}{Theorem}[section]
\newtheorem{lemma}[theorem]{Lemma}
\newtheorem{proposition}[theorem]{Proposition}
\newtheorem{corollary}[theorem]{Corollary}
\theoremstyle{definition}
\newtheorem{definition}[theorem]{Definition} 
\newtheorem{example}[theorem]{Example}
\newtheorem{remark}[theorem]{Remark}
\numberwithin{equation}{section}

\begin{document}
	

	\renewcommand{\bf}{\bfseries}
	\renewcommand{\sc}{\scshape}

	\title[Smyth complete  real-enriched categories]%
	{Smyth complete  real-enriched categories}
	
	\author{Junche Yu}
	\address{School of Mathematical Sciences, Beihang University, Beijing, China}
	\email{cqyjc@icloud.com}

	\author{Dexue Zhang}
	\address{School of Mathematics, Sichuan University, Chengdu, China}
	\email{dxzhang@scu.edu.cn}
	
	%

	
	\begin{abstract}  This paper investigates Smyth completeness of categories enriched over a quantale obtained by equipping the unit interval of real numbers with a continuous t-norm. A real-enriched category is Smyth-complete if each of its forward Cauchy nets has a unique limit in the open ball topology of its symmetrization. It is demonstrated that Smyth completeness can be characterized as a categorical property and as a real-valued topological property. Explicitly, it is shown that a real-enriched category is Smyth complete if and only if it is separated and all of its ideals are representable, if and only if its Alexandroff real-valued topology is sober. 
	\end{abstract}
	
	\maketitle
	
	\section{\bf Introduction}
	Smyth completeness originated in the works of Smyth \cite{Smyth88,Smyth94} on quasi-uniform spaces that aimed to provide a common framework for the domain approach and the metric space approach to semantics in computer science. For quasi-metric spaces, Smyth completeness is postulated as follows: a quasi-metric space $X$ is Smyth complete if every forward Cauchy net of $X$ converges uniquely in the symmetrization of $X$ \cite{Goubault,KS2002}. Smyth complete spaces, Smyth complete quasi-metric spaces in particular, have received attention in different areas; to name a few: quasi-Polish spaces  \cite{Brecht}, domain theory \cite{Goubault,KS2002}, quasi-uniform spaces \cite{KSF,Sunderhauf97}, approach spaces \cite{LiZ18a}. 
	
	This paper investigates Smyth completeness for real-enriched categories. A real-enriched category is a category with enrichment coming from a quantale obtained by equipping the interval $[0,1]$ with a continuous t-norm.  Such categories are among the core objects in quantitative domain theory (see e.g. \cite{AW2011,BBR98,Goubault,HW2011,HW2012}) and monoidal topology (see e.g. \cite{GH2013,Hof2007,HST2014,LiZ18a,Lowen2015}). 
	The postulation of Smyth completeness for quasi-metric spaces can be directly extended to the real-enriched context: a real-enriched category is Smyth complete if each of its forward Cauchy nets has a unique limit in the open ball topology of its symmetrization. This paper   presents two characterizations of such categories. The first is purely in terms of categorical structures, the second  in terms of real-valued topological structures. Both characterizations do not make use of the symmetrization  of real-enriched categories.  
	
	The contents are arranged as follows. Section \ref{Smyth completeness} contains some basic notions, including open ball topology, bilimits and Yoneda limits of nets,  and Smyth completeness for real-enriched categories. Section \ref{As a categorical property} characterizes Smyth completeness in terms of categorical structures. It is shown that a real-enriched category is Smyth complete if and only if it is separated and all of its ideals (which may be viewed as ind-objects in the real-enriched context) are representable. This is in parallel to Lawvere's observation in  \cite{Lawvere1973} that a metric space is complete if and only if, as a category enriched over the quantale $([0,\infty]^{\rm op},+,0)$, it is separated and all of its Cauchy weights are representable.  Section \ref{As a real-valued topological property} characterizes Smyth completeness in terms of real-valued topological properties: a real-enriched category is Smyth complete if and only if its Alexandroff real-valued topology is sober. 
	Section \ref{Smyth completable} concerns properties of Smyth completable real-enriched categories, i.e., those categories that can be embedded in a Smyth complete one.
	
	\section{\bf Smyth completeness of real-enriched categories} \label{Smyth completeness}
	
	A \emph{continuous t-norm} \cite{Klement2000} on $[0,1]$ is a  continuous function \[\with\colon[0,1]\times[0,1]\to[0,1],\] called multiplication, such that  $([0,1],\with,1)$ is a unital quantale in the sense of \cite{Rosenthal1990}. Explicitly,  \begin{enumerate}[label=\rm(\roman*)] \item $([0,1],\&,1)$ is a monoid; \item   $\&$ is continuous and  non-decreasing on each variable. \end{enumerate}
	
	Given a continuous t-norm $\with$, the binary function $$\ra\colon[0,1]\times[0,1]\to[0,1],$$ given by $$x\ra z=\sup\{y\in[0,1]\mid x\with y\leq z\},$$ is called the \emph{implication operator} of $\with$. In the language of category theory, $\ra$ is the \emph{internal hom} of  $([0,1],\with,1)$  when viewed as a symmetric monoidal closed category. 
	
	\begin{example}
		Basic  continuous t-norms and their implication operators:
		\begin{enumerate}[label={\rm(\roman*)}]  
			\item  The G\"{o}del t-norm  \[ x\with y= \min\{x,y\}; \quad x\ra z=\begin{cases}
				1 &x\leq z,\\
				z &x>z.
			\end{cases} \]  
			\item  The product t-norm   \[ x\with  y=x\cdot y; \quad x\ra z=\begin{cases}
				1 &x\leq z,\\
				z/x &x>z.
			\end{cases} \]  
			The quantale $([0,1],\cdot,1)$ is isomorphic to  $([0,\infty]^{\rm op},+,0)$.
			
			\item  The {\L}ukasiewicz t-norm \[ x\with  y=\max\{0,x+y-1\}; \quad x\ra z=\min\{1-x+z,1\}. \] 
	\end{enumerate} \end{example}

	Suppose $\with $ is a continuous t-norm. An element $p\in [0,1]$ is   \emph{idempotent}  if $p\with p=p$.
	It is known (see e.g. \cite[Proposition 2.3]{Klement2000}) that if $p$ is idempotent, then $x\with y=  x\wedge y $ whenever $x\leq p\leq y$. From this it follows that for any idempotent elements $p, q$    with $p<q$,  the restriction of $\with $ to $[p,q]$  makes $([p,q],\with,q)$ into a commutative and unital quantale.
	
	A continuous t-norm on $[0,1]$ is \emph{Archimedean} 
	if it has no idempotent element  other than $0$ and $1$. It is  known that a continuous Archimedean t-norm  is either isomorphic to product t-norm or to the {\L}ukasiewicz t-norm, see e.g. \cite{Klement2000}. A fundamental result on continuous t-norms says that every continuous t-norm is an ordinal sum of the product t-norm and the {\L}ukasiewicz t-norm. 
	
	\begin{theorem}{\rm(\cite{Klement2000,Mostert1957})} \label{ordinal sum} For each continuous t-norm $\with$ on $[0,1]$, there exists a countable family of pairwise disjoint open intervals $\{(a_i,b_i)\}_{i\in J}$ of $[0,1]$ such that \begin{enumerate}[label=\rm(\roman*)] \item for each $i\in J$, both $a_i$ and $b_i$ are idempotent and the restriction of $\&$ on the closed interval $[a_i,b_i]$ is either isomorphic to the product t-norm or to the \L ukasiewicz t-norm; \item  $x\with y=\min\{x,y\}$ for any $(x,y)$ outside the union of the squares $[a_i,b_i]^2$. \end{enumerate}  
	\end{theorem}
	
	In this paper $\&$ always denotes a continuous t-norm, unless otherwise specified. Categories enriched over the quantale $([0,1],\with,1)$ are called real-enriched categories, because the enrichment are real numbers. 
	
	\begin{definition}
		A real-enriched category is a pair $(X, \alpha)$, where $X$ is a set,   $\alpha\colon X\times X\to[0,1]$ is a function  such that \begin{enumerate}[label=\rm(\roman*)]  \item   $\alpha(x,x)= 1$ for all $x\in X$;     \item   $ \alpha(y,z)\with \alpha(x,y)\leq \alpha(x,z)$ for all $x,y,z\in X$.  
	\end{enumerate}  \end{definition}
	
	To ease notations, we often  write $X$ for $(X, \alpha)$ and  $X(x,y)$ for $\alpha(x,y)$.  
	
	Suppose $(X,\alpha)$ is a real-enriched category. \begin{itemize} \item The \emph{opposite} of $(X,\alpha)$ refers to the  category $(X, \alpha^{\rm op})$, where   $\alpha^{\rm op}(x,y)=\alpha(y,x)$. \item The  \emph{symmetrization}  of  $(X,\alpha)$  refers to the symmetric (in an evident sense) real-enriched category $(X,S(\alpha))$, where $S(\alpha)(x,y)=\min\{\alpha(x,y),\alpha(y,x)\}$.  \item  
		The \emph{underlying order} of  $X$ refers to the reflexive and transitive relation   $\sqsubseteq$ on $X$ given by $x\sqsubseteq y$ if $X(x,y)=1$.  
		We write $X_0$ for the ordered set $(X,\sqsubseteq)$. \end{itemize}
	
	Two elements $x$ and $y$ of a real-enriched category $X$ are isomorphic if $X(x,y)=1=X(y,x)$.  We say that $X$ is   \emph{separated} if its isomorphic elements are identical. In other words, $X$ is separated if the underlying order $\sqsubseteq$ is anti-symmetric.
	
	\begin{example} The pair $([0,1],\alpha_L)$ is a separated real-enriched category, where $\alpha_L (x,y)= x\ra y$.  The opposite  of $([0,1],\alpha_L)$ is given by $([0,1],\alpha_R)$, where $\alpha_R (x,y)=y\ra x.$    We shall often write $\sV$ for   $([0,1],\alpha_L)$,   and $\sV^{\rm op}$ for   $([0,1],\alpha_R)$, respectively.   
		
		Generally, for each set  $X$ the pair $([0,1]^X,\sub_X)$ is a separated real-enriched category, where   for all $\lam,\mu\in [0,1]^X$,
		$$\sub_X(\lam,\mu) =\inf_{x\in X}(\lam(x)\ra \mu(x)). $$   When we view $\lam$ and $\mu$ as fuzzy  sets, the value $\sub_X(\lam,\mu)$ measures the degree that $\lam$ is contained in $\mu$, so the category $([0,1]^X,\sub_X)$ is also  called the {enriched powerset} of $X$.  \end{example} 
	
	Suppose $X$ and $Y$ are real-enriched categories. A functor   $f\colon X\to Y$  is a map such that  $X(x,y)\leq Y(f(x),f(y))$ for all $x,y\in X$. The category of real-enriched categories and functors is denoted by \[[0,1]\text{-}\mathbf{Cat}.\]  
	
	Suppose $X$ is a real-enriched category. For each $x\in X$ and $r<1$, the set $$B(x,r)\coloneqq\{y\in X\mid X(x,y)>r\}$$ is called the open ball of $X$ with center $x$ and radius $r$. The collection $\{B(x,r)\mid x\in X, r<1\}$  is a base for a topology on $X$, the resulting topology is called the \emph{open ball topology} of $X$. It is clear that the open ball topology of the symmetrization of $X$ is the least common refinement of the open ball topology of $X$ and that of $X^{\rm op}$.
	
	\begin{proposition} \label{limit in open ball top} A net $\{x_i\}_{i\in D}$ of a real-enriched category $X$ converges to $x$ in the open ball topology if and only if $\bv_{i\in D}\inf_{j\geq i}X(x,x_j)=1$. \end{proposition}
	
	The following definition collects the notions of Cauchy nets, forward Cauchy nets, bilimits and Yoneda limits of nets that are scattered in the literature, see e.g. \cite{BBR98,Goubault,KS2002,Wagner97}.
	\begin{definition}
		Suppose $\{x_i\}_{i\in D}$ is a net of a real-enriched category $X$.  \begin{enumerate}[label=\rm(\roman*)]  \item $\{x_i\}_{i\in D}$   is   Cauchy (also called  biCauchy) if \[\sup_{i\in D}\inf_{j,k\geq i}X(x_j,x_k)= 1.\] \item $\{x_i\}_{i\in D}$ is  forward Cauchy if \[\sup_{i\in D}\inf_{k\geq j \geq i}X(x_j,x_k)= 1.\]  
			\item A bilimit  of $\{x_i\}_{i\in D}$ is an element  $a$ of $X$ such that  for all $x\in X$, \[\sup_{i\in D}\inf_{j\geq i}X(x,x_j)=X(x,a),\quad  \sup_{i\in D}\inf_{j\geq i}X(x_j,x)=X(a,x).  \] 
			\item A Yoneda limit  of $\{x_i\}_{i\in D}$ is an element  $b$ of $X$ such that  for all $x\in X$, \[  \sup_{i\in D}\inf_{j\geq i}X(x_j,x)=X(b,x).  \] 
	\end{enumerate} \end{definition}
	
	Every Cauchy net is forward Cauchy. Actually, a net of $X$ is Cauchy if and only if it is forward Cauchy both in $X$ and in $X^{\rm op}$, if and only if it is Cauchy in the symmetrization of $X$. A bilimit is obviously a Yoneda limit, but not conversely. A net has at most one Yoneda limit up to isomorphism,   at most one bilimit up to isomorphism.
	
	\begin{lemma}\label{bilimit of Cauchy net} Suppose $\{x_i\}_{i\in D}$ is a Cauchy net and $a$ is an element of a real-enriched category $X$. Then the following   are equivalent:  \begin{enumerate}[label={\rm(\arabic*)}]   \item $a$ is a bilimit of $\{x_i\}_{i\in D}$. \item $a$ is a Yoneda limit of $\{x_i\}_{i\in D}$. \item   $\sup_{i\in D}\inf_{j\geq i}X(a,x_j)= 1$ and $ \sup_{i\in D}\inf_{j\geq i}X(x_j,a)= 1$. 
			\item $\{x_i\}_{i\in D}$ converges to $a$ in the open ball topology of its symmetrization. \end{enumerate} \end{lemma} 
	
	\begin{proof} $(1)\Rightarrow(2)$   trivial.
		
		$(2)\Rightarrow(3)$ We only need to check that   $\sup_{i\in D}\inf_{j\geq i}X(a,x_j)= 1.$ For this we calculate:  \begin{align*}\sup_{i\in D}\inf_{j\geq i}X(a,x_j) &= \sup_{i\in D}\inf_{j\geq i} \sup_{h\in D}\inf_{k\geq h} X(x_k,x_j)\\ &\geq \sup_{i\in D}\inf_{k,j\geq i} X(x_k,x_j) \\ &= 1.\end{align*} 
		
		$(3)\Leftrightarrow(4)$ This follows from Proposition \ref{limit in open ball top} and  the fact  that the requirement in (3) is equivalent to that $$\sup_{i\in D}\inf_{j\geq i}\min\{X(a,x_j),X(x_j,a)\}= 1.$$
		
		$(3)\Rightarrow(1)$ For all $x\in X$, we calculate: \begin{align*} X(x,a)&= \Big(\sup_{i\in D}\inf_{j\geq i}X(a,x_j)\Big)\with X(x,a) \\ &\leq  \sup_{i\in D}\inf_{j\geq i} X(a,x_j) \with X(x,a)  \\ &\leq  \sup_{i\in D}\inf_{j\geq i}X(x,x_j)   \\ &=  \Big( \sup_{i\in D}\inf_{j\geq i}X(x_j,a)\Big) \with \Big(\sup_{i\in D}\inf_{j\geq i}X(x,x_j)\Big) \\ &= \sup_{i\in D}\inf_{j\geq i} X(x_j,a)\with X(x,x_j)    \\ &\leq X(x,a),\end{align*} which implies   $\sup_{i\in D}\inf_{j\geq i}X(x,x_j)=X(x,a)$.  The other equality is verified likewise. So  $a$ is a bilimit of $\{x_i\}_{i\in D}$.  \end{proof}

	Proposition \ref{limit in open ball top}  implies that if a net of $X$ converges in the open ball topology of its symmetrization, then it is a Cauchy net of $X$.
	Now we introduce the central notion of this paper.
	
	\begin{definition}A real-enriched category  is Smyth complete if each pf its forward Cauchy nets  converges uniquely in the open ball topology of its symmetrization. \end{definition}
	
	Smyth completeness originated in the works of Smyth \cite{Smyth88,Smyth94} on quasi-uniform spaces. The above postulation is a direct extension of that for Smyth complete quasi-metric spaces in \cite{Goubault,KS2002}.
	
	\begin{example}This example concerns Smyth completeness of the real-enriched categories  ${\sf V}=([0,1],\alpha_L)$ and ${\sf V}^{\rm op}=([0,1],\alpha_R)$.  We distinguish three cases.  
		
		{\bf Case 1}. The continuous t-norm is not Archimedean. Pick some $a\in[0,1]$ with $0<a\leq a^+<1$, where $a^+$ is the least idempotent element in $[a,1]$. Since $\bv_{x<a}(a\ra x)\leq a^+$, it follows that $[a,1]$ is open in $\sV$ and $[0,a]$ is open in $\sV^{\rm op}$, then $\{a\}$ is open in the symmetrization of $\sV$.  Since $\{a-1/n\}_n$ is forward Cauchy in $\sV$ and $\{a+1/n\}_n$ is forward Cauchy in $\sV^{\rm op}$, neither $\sV$ nor $\sV^{\rm op}$ is Smyth complete in this case.
		
		{\bf Case 2}. The continuous t-norm is isomorphic to the {\L}ukasiewicz t-norm. In this case, the open ball topology of $\sV$ is $\{\varnothing,[0,1]\}\cup\{(a,1]\mid a<1\}$, the open ball topology of $\sV^{\rm op}$ is $\{\varnothing,[0,1]\}\cup\{[0,a)\mid a>0\}$.  It is easily seen that a net $\{x_i\}_{i\in D}$ is forward Cauchy in $\sV$ if and only if it is  Cauchy in the usual sense, if and only if it  is forward Cauchy in $\sV^{\rm op}$, so, both $\sV$ and $\sV^{\rm op}$ are Smyth complete. 
		
		{\bf Case 3}. The continuous t-norm is isomorphic to the product t-norm. In this case, the open ball topology of $\sV$ is $\{\varnothing,[0,1]\}\cup\{(a,1]\mid a<1\}$, the open ball topology of $\sV^{\rm op}$ is $\{\varnothing,\{0\},[0,1]\}\cup\{[0,a)\mid a>0\}$. Since the sequence  $\{1/n\}_{n\geq1}$ is forward Cauchy  in $\sV^{\rm op}$, but not convergent in its symmetrization, $\sV^{\rm op}$ is not Smyth complete. It is not hard to verify that a net $\{x_i\}_{i\in D}$ is forward Cauchy in $\sV$ if and only if it is eventually constant or converges (in the usual sense) to some element other than $0$, so $\sV$ is Smyth complete.   \end{example} 
	
	\begin{proposition}\label{FC in Smyth is C} If a real-enriched category $X$ is Smyth complete, then 
		it is separated and all of its forward Cauchy nets are Cauchy.  \end{proposition}

	\section{\bf As a categorical property} \label{As a categorical property}
	
	In this section we demonstrate that Smyth completeness, though postulated in terms of topological properties, can be characterized purely in terms of categorical structures. Precisely, a real-enriched category is Smyth complete if and only if it is separated and all of its ideals (to be defined later) are representable. To state and prove the result (Theorem \ref{Smyth via bilimit}), we need a few basic notions and results in the theory of quantale-enriched categories. To fix notations and to make the paper more or less self-contained, we begin  with some basic ideas about quantale-enriched categories: distributors, colimits, Cauchy weights, and etc. Readers familiar with  quantale-enriched categories or the general theory of enriched categories \cite{Borceux1994,Kelly} may skip the materials before Proposition \ref{colimit =bilimit for Cauchy net}.
	
	\begin{definition}
		Suppose $X$ and $Y$ are real-enriched categories. A distributor $\phi$ from $X$ to $Y$, written $\phi\colon X\oto Y$, is a map $\phi\colon X\times Y\to [0,1]$ such that \begin{enumerate}[label={\rm(\roman*)}]  
			\item for all $x_1,x_2\in X$ and $y\in Y$, $\phi(x_2,y)\with X(x_1,x_2)\leq \phi(x_1,y)$;  \item for all $x\in X$ and $y_1,y_2\in Y$,   $Y(y_1,y_2)\with\phi(x,y_1)\leq \phi(x,y_2)$.  \end{enumerate}      \end{definition} 
	
	Given a distributor $\phi\colon X\oto Y$, the distributor  $$\phi^{\rm op}\colon Y^{\rm op}\oto X^{\rm op}, \quad \phi^{\rm op} (y,x)=\phi(x,y)$$ is called the opposite of $\phi$.  Distributors can be composed: for $\phi\colon X\oto Y$ and $\psi\colon Y\oto Z$, $$\psi\circ\phi(x,z)=\bv_{y\in Y}\psi(y,z)\with\phi(x,y).$$ With real-enriched categories as objects and distributors as morphisms we have a category, indeed a quantaloid (see \cite{Rosenthal1996,Stubbe2005} for definition)  \[[0,1]\text{-}\mathbf{Dist},\] called the category of  distributors \cite{Stubbe2005}. For each real-enriched category $X$, the identity distributor on $X$ is its categorical structure, which will also be denoted by $X$.
	
	Since $[0,1]\text{-}\mathbf{Dist}$ is a quantaloid, for each distributor  $\phi\colon X\oto Y$,   the map \[-\circ \phi\colon [0,1]\text{-}\mathbf{Dist}(Y,Z)\to [0,1]\text{-}\mathbf{Dist}(X,Z)\] has a right adjoint \[-\swarrow \phi\colon [0,1]\text{-}\mathbf{Dist}(X,Z)\to[0,1]\text{-}\mathbf{Dist}(Y,Z).\]  
	$$\bfig \ptriangle(0,0)/>`>`>/<520,500>[X`Y`Z;\phi`\xi`\xi\swarrow \phi]
	\place(260,500)[\circ] \place(265,255)[\circ] \place(0,255)[\circ] \place(150,350)[\geq]
	\efig$$
	Likewise, for each distributor  $\psi\colon Y\oto Z$, the map \[\psi\circ -\colon [0,1]\text{-}\mathbf{Dist}(X,Y)\to [0,1]\text{-}\mathbf{Dist}(X,Z)\] has a right adjoint \[\psi\searrow -\colon [0,1]\text{-}\mathbf{Dist}(X,Z)\to[0,1]\text{-}\mathbf{Dist}(X,Y). \] 
	$$\bfig \qtriangle(0,0)/>`<-`<-/<520,500>[Y`Z`X;\psi`\psi\searrow \xi`\xi]
	\place(260,500)[\circ] \place(257,255)[\circ] \place(520,255)[\circ] \place(350,350)[\leq]
	\efig$$ 
	
	Given a pair of distributors $\phi\colon Y\oto X$ and $\psi\colon X\oto Y$, we say   $\phi$ is right adjoint to   $\psi$,  or $\psi$ is left adjoint to $\phi$, and write $\psi\dashv\phi$, if $X\leq \phi\circ\psi$ and $\psi\circ\phi\leq Y$. 
	
	The following simple lemma is  useful in the calculus of distributors.
	
	\begin{lemma} \label{adjoint_arrow_calculation}  
		Suppose   $\psi\colon X\oto Y$ is left adjoint to $\phi\colon Y\oto X$.  Then
		\begin{enumerate}[label={\rm(\roman*)}]  
			\item For each distributor $\xi\colon Y\oto Z$, $\xi\circ \psi=\xi\swarrow \phi$. In particular, $\psi=Y\swarrow\phi$.
			
			\item  For each distributor $\lambda\colon W\oto Y$, $\phi\circ \lambda =\psi\searrow \lambda$. In particular, $\phi=\psi\searrow Y$. \end{enumerate}  \end{lemma}

	
	Therefore, a distributor has at most one right (left, resp.) adjoint. So, we  speak of \emph{the right (left, resp.) adjoint} of a distributor.
	
	Fr each functor $f\colon X\to Y$, the distributor  \[f_*\colon X\oto Y, \quad f_*(x,y)=Y(f(x),y)\] is called the graph of $f$, the distributor \[f^*\colon Y\oto X, \quad f^*(y,x)=Y(y,f(x))\] is called the  cograph  of $f$. It is known that the graph  $f_* $ is left adjoint to the cograph $f^* $. Furthermore, $f\colon X\to Y$ is fully faithful (i.e., $X(x_1,x_2)=Y(f(x_1),f(x_2))$ for all $x_1,x_2\in X$) if and only if $f^*\circ f_*=X$.

	By a \emph{weight} of a real-enriched category $X$ we mean a distributor of the form $\phi\colon X\oto \star$, where $\star$ is the terminal real-enriched category. Or equivalently, a weight of $X$ is a functor $\phi\colon X^{\rm op}\to \sV$, where $\sV=([0,1],\alpha_L)$.  All weights of $X$ constitute a separated real-enriched category $\mathcal{P}X$ with $$\mathcal{P}X(\phi_1,\phi_2) =\phi_2\swarrow\phi_1=\sub_X(\phi_1,\phi_2) .$$ 
	
	The assignment $X\mapsto\CP X$ is functorial.  For each functor $f\colon X\to Y$,   it is easily verified that  $$f^\rightarrow\colon\CP  X\to\CP Y, \quad \phi\mapsto \phi\circ f^*$$ is a functor between real-enriched categories.  Assigning to each $f\colon X\to Y$ the functor $f^\ra$  defines a functor $$\CP\colon[0,1]\text{-}\mathbf{Cat}\to[0,1]\text{-}\mathbf{Cat}. $$ 
	
	Dually,  distributors of the form $\star\oto X$ are called  \emph{coweights} of $X$. A coweight of $X$ is essentially a  functor $\psi\colon X\to \sV$ with $\sV=([0,1],\alpha_L)$. All coweights of $X$ constitute a real-enriched category $\CPd X$ with $$\CPd X(\psi_1,\psi_2) =\psi_2\searrow\psi_1=\sub_X(\psi_2,\psi_1) .$$     
	
	For each element $a$ of a real-enriched category $X$, the assignment $x\mapsto X(x,a)$ defines a weight of $X$, the assignment $x\mapsto X(a,x)$ defines a coweight of $X$. Weights of the form $X(-,a)$ and   coweights of the form $X(a,-)$ are said to be  \emph{representable}.  The  coweight $X(a,-)$ is left adjoint to the   weight $X(-,a)$.

	The following lemma, a special case of the Yoneda lemma in the theory of enriched categories, implies that for each real-enriched category $X$,  both \[\sy\colon X\to\CP X, \quad a\mapsto X(-,a)\] and \[\syd\colon X\to\CPd  X, \quad a\mapsto X(a,-)\] are fully faithful functors, they are called the Yoneda embedding and the co-Yoneda embedding, respectively.  
	
	\begin{lemma}[Yoneda lemma] Suppose $X$ is a real-enriched category and $a\in X$. \begin{enumerate}[label={\rm(\roman*)}]  
			\item For each   $\phi\in\CP X$, $\CP X(\sy(a),\phi)=\phi(a)$.  
			\item   For each   $\psi\in\CPd  X$, $\CPd  X(\psi,\syd(a))=\psi(a)$. \end{enumerate}   \end{lemma}
	
	A \emph{colimit} of a weight $\phi\colon X\oto\star $ is an element $\colim\phi$ of $X$ such that for all $x\in X$, \[X(\colim\phi,x)=\CP X(\phi,\sy(x)).\]
	It is clear that each weight has, up to isomorphism, at most one colimit. So, we shall speak of \emph{the} colimit of a weight. Since \[\CP X(\phi,\sy(x))=X(-,x)\swarrow\phi = (X\swarrow\phi)(x) \] for all $x\in X$,  then $\phi$ has a colimit if and only if the coweight $X\swarrow\phi$  is representable.  
	
	For each functor $f\colon K\to X$  and each weight $\phi$  of $K$, the colimit of the weight $\phi\circ f^*$, when exists, is called the \emph{colimit of $f$ weighted by $\phi$}. In particular,  $\colim\phi$ is  the colimit of the identity functor $X\to X$ weighted by $\phi$.  We  write $\colim_\phi f$ for the colimit of $f$ weighted by $\phi$.  
	
	\begin{example} \label{cpt as colimit} Let $\psi\colon X\to\sV$ be a functor, where $\sV=([0,1],\alpha_L)$. Then for each weight $\phi$   of   $X$,  the colimit of $\psi$ weighted by $\phi$ exists and is equal to the composite $\phi\circ\psi$ when $\psi$ is viewed as a distributor $\star\oto X$; that is, $\colim_\phi\psi =\phi\circ\psi$. 
	\end{example}

	A real-enriched category $X$ is   \emph{cocomplete} if every weight of $X$ has a colimit.  It is clear that a real-enriched category $X$ is cocomplete if and only if all weighted colimits of functors to $X$ exist; that means, for each functor $f\colon K\to X$ and each weight $\phi$ of $K$, the colimit of $f$ weighted by $\phi$ exists. 
	
	The dual notion of colimit is limit. Precisely, a  \emph{limit} of a coweight $\psi\colon \star\oto X$ is an element $\lim\psi$ of $X$ such that for all $x\in X$, \[\CPd  X(\syd(x),\psi)=X(x,\lim\psi).\] Each coweight has at most one limit up to isomorphism. Since \[\CPd  X(\syd(x),\psi)= \psi\searrow X(x,-) = (\psi\searrow X)(x),\]   $\psi$ has a limit if and only if the weight $\psi\searrow X$  is representable. 
	
	It is readily verified that for each coweight $\psi\colon \star\oto X$ of a real-enriched category $X$, an element  $a$ of $X$ is a limit of $\psi$ if and only if $a$ is a colimit of the weight $\psi^{\rm op}$ of $X^{\rm op}$.

	For each functor $f\colon K\to X$  and each coweight $\psi$  of $K$, the limit of the coweight $f_*\circ\psi$, when exists, is called the \emph{limit of $f$ coweighted by $\psi$}. In particular, $\lim\psi$ is the limit of the identity functor $X\to X$ coweighted by $\psi$.  We  write $\lim_\psi f$ for the limit of $f$ coweighted by $\psi$.  
	
	\begin{example} \label{sub as limit} Let $\xi\colon X\to\sV$ be a functor, where $\sV=([0,1],\alpha_L)$.  Then for each coweight $\psi$  of  $X$, the limit of $\xi$ coweighted by $\psi$ exists and is given by ${\lim}_\psi\xi=\sub_X(\psi,\xi).$ \end{example} 
	

	A real-enriched category $X$ is \emph{complete} if every coweight of $X$ has a limit. It is clear that $X$ is complete if and only if for any functor $f\colon K\to X$ and any coweight $\psi$ of $K$, the limit of $f$ coweighted by $\psi$ exists.
	
	Suppose $f\colon X\to Y$ is a functor. We say that  $f$  preserves the colimit of a weight $\phi$ of $X$ if $f(\colim\phi)$ is a colimit of the weight $\phi\circ f^*$ whenever $\colim\phi$ exists. Likewise,  we say that $f$ preserves the limit of a coweight $\psi$ of $X$ if $f(\lim\psi)$ is a limit of the coweight $f_*\circ\psi$ whenever $\lim\psi$ exists. We say that  $f$    \emph{preserves colimits} if it preserves all colimits that exist. Likewise,  $f$  \emph{preserves limits} if it preserves all  limits that exist. 
	
	
	Suppose $X,Y$ are real-enriched categories, $f\colon X\to Y$ and $g\colon Y\to X$ are functors.  We say that $f$ is left adjoint to $g$, or $g$ is right adjoint to $f$, and write $f\dashv g$, if $$Y(f(x),y)=X(x,g(y))$$ for all $x\in X$ and $y\in Y$. The pair $(f,g)$ is  called an (enriched) adjunction. This is a special case of  adjunction in the theory of enriched categories. 
	
	Here is a typical example of adjunctions. For each functor $f\colon X\to Y$ between real-enriched categories, the functor  $f^\rightarrow\colon\CP  X\to\CP Y$ sending a weight $ \phi$ of $X$ to $\phi\circ f^*$ is left adjoint to the  functor $f^\la\colon\CP  Y\to \CP X$ sending a weight $\psi$ of $Y$ to  $\psi\circ f_*$. 
	
	It is known that every left adjoint preserves colimits, every right adjoint preserves limits. Conversely,   if $f\colon X\to Y$ preserves colimits and $X$ is cocomplete, then $f$ is a left adjoint.  
	
	Suppose $X$ is a real-enriched category. For each $x\in X$ and $r\in [0,1]$, the \emph{tensor of $r$ with $x$}, denoted by $r\otimes  x$, is an element of $X$ such that for all $y\in X$, \[X(r\otimes  x,y)= r\ra X(x,y).\] The tensor $r\otimes  x$ exists if and only if the weight \[r\with\sy(x)\colon X\oto\star, \quad z\mapsto r\with X(z,x)   \] has a colimit, in which case  $r\otimes  x=\colim(r\with\sy(x))$. 
	
	For each $y\in X$ and  $r\in [0,1]$, the \emph{cotensor of $r$ with} $y$, denoted by $r\multimap y$,  is an element of $X$ such that for all $x\in X$, \[X(x,r\multimap y)=r\ra X(x,y).\] The cotensor $r\multimap y$ exists if and only if the coweight \[r\with\syd(y) \colon \star\oto X, \quad z\mapsto r\with X(y,z) \] has a limit, in which case  $r\multimap y=\lim(r\with\syd(y))$. 
	
	A real-enriched category $X$ is   \emph{tensored} if the tensor $r\otimes  x$ exists for all $x\in X$ and $r\in [0,1]$; $X$ is   \emph{cotensored} if the cotensor $r\multimap y$ exists for all $y\in X$ and $r\in [0,1]$. 
	It is clear that $X$ is tensored if and only if   $X^{\rm op}$  is cotensored.

	For each real-enriched category $X$, the categories  $\CP X$ and $\CPd X$ are both tensored and cotensored. In $\CP X$, the tensor and cotensor of $r\in [0,1]$ with $\phi\in\CP X$ are   $r\with\phi$ and $r\ra\phi$, respectively. In $\CPd X$, the tensor and cotensor of $r\in [0,1]$ with $\psi\in\CPd X$ are  $r\ra  \psi$ and $r\with\psi $, respectively.  
	
	\begin{theorem}\label{characterizing completeness} {\rm (\cite{Stubbe2005,Stubbe2006})} For each real-enriched category $X$, the following  are equivalent: \begin{enumerate}[label={\rm(\arabic*)}]  
			\item  $X$ is cocomplete. \item  $X$ is complete.  \item The Yoneda embedding $\sy\colon X\to\CP X$ has a left adjoint. \item The coYoneda embedding $\syd\colon X\to\CPd X$ has a right adjoint.
			\item $X$ is tensored, cotensored, and every subset of $X$ has a join in the ordered set $X_0$.
	\end{enumerate}\end{theorem}
	
	From the above theorem it follows that if  $X,Y$ are cocomplete real-enriched categories, then a functor  $f\colon X\to Y$ is a left adjoint if and only if   $f\colon X_0\to Y_0$ preserves joins and $f$ preserves tensors in the sense that $f(r\otimes  x)=r\otimes  f(x)$ for all $r\in [0,1]$ and $x\in X$.

	\begin{corollary}\label{calculation of sup by tensors} Suppose $X$ is a complete real-enriched category. Then, for each functor  $f\colon K\to X$ and each weight $\phi$ of $ K$, the colimit of $f$ weighted by $\phi$  is given by  the join of  $\{\phi(z)\otimes  f(z)\mid z\in K\}$ in $X_0$. 
	\end{corollary}
	
	\begin{proof}  Since $\phi\circ f^*$ is the join  in   $(\CP X)_0$ of the subset $$\{\phi(z)\with X(-,f(z))\mid z\in K\},$$    the functor $\colim\colon\CP X\to X$  preserves colimits, it follows that \begin{align*}{\colim}_\phi f&= \colim  \phi\circ f^* \\ & = \sup_{z\in K}\colim(\phi(z)\with X(-,f(z)))\\ & = \sup_{z\in K} \phi(z)\otimes  f(z) . \qedhere \end{align*} 
	\end{proof} 
	
	\begin{definition}Suppose $\phi\colon X\oto\star$ is a weight of a real-enriched category  $X$. We say that $\phi$ is Cauchy if it is a right adjoint,  as a distributor. That means, there is a distributor $\psi\colon\star\oto X$ such that $\phi\circ\psi\geq1$ and $\psi(y)\with\phi(x)\leq X(x,y)$ for all $x,y\in X$. \end{definition}
	
	Suppose $f\colon X\to Y$ is a functor and $\phi $ is a Cauchy weight   of $X$.  Then, $\phi\circ f^*$, being a composite of right adjoints,  is a right adjoint, hence a Cauchy weight of $Y$. 
	
	If a Cauchy weight $\phi$  of a real-enriched category $X$ has a colimit, then $\phi$ is representable and its colimit is preserved by any  functor. To see this, suppose $a$ is a colimit of a Cauchy weight $\phi\colon X\oto*$. Then   $X\swarrow\phi=X(a,-)$. Since $X\swarrow\phi$ is   left adjoint to  $\phi$ by Lemma \ref{adjoint_arrow_calculation}, it follows that $\phi=X(a,-)\searrow X =X(-,a)$, showing that $\phi$ is representable. The second conclusion follows from  that for any functor $f\colon X\to Y$ and any weight $\phi$ of $X$, the weight $\phi\circ f^*$ is representable whenever so is $\phi$.  
	
	The following proposition relates colimits of Cauchy weights to bilimits of Cauchy nets.
	\begin{proposition}\label{colimit =bilimit for Cauchy net} For each real-enriched category $X$, the following are equivalent:
		\begin{enumerate}[label={\rm(\arabic*)}]   \item Every Cauchy weight of $X$ has a colimit. \item Every Cauchy net of $X$ has a bilimit. \item  Every Cauchy sequence of $X$ has a bilimit.  \end{enumerate} \end{proposition}
	
	\begin{proof}$(1)\Rightarrow(2)$ Suppose   $\{x_i\}_{i\in D}$ is a Cauchy net of $X$. Then the weight \[\phi =\sup_{i\in D}\inf_{j\geq i}X(-,x_j)\] is Cauchy with a left adjoint given by the coweight  \[ \psi=\sup_{i\in D}\inf_{j\geq i}X(x_j,-). \]   Actually, a slightly more general conclusion will be proved in Lemma \ref{Cauchy net implies Cauchy weight}. Let $a$ be a colimit of the Cauchy weight $\phi$. Then $\phi(x)=X(x,a)$ and $\psi(x)=X(a,x)$  for all $x\in X$, which implies that $a$ is a bilimit of $\{x_i\}_{i\in D}$.  
		
		$(2)\Rightarrow(3)$ Trivial.  
		
		$(3)\Rightarrow(1)$ The proof is essentially the same  as that in \cite{Lawvere1973} for metric spaces, we include it here for convenience of the reader.
		Let $\phi$ be a Cauchy weight of $X$, with a left adjoint  $\psi$. Then $\sup_{x\in X} \phi(x)\with \psi(x) \geq 1$ and $\psi(y)\with\phi(x)\leq X(x,y)$ for all $x,y\in X$. 
		
		For each $n\geq1$, pick some $x_n$ such that \[\phi(x_n)\with \psi(x_n) \geq  1-1/n.\] Since for all $n,m\geq1$ we have \[X(x_n,x_m)\geq \psi(x_m)\with\phi(x_n)\geq (1-1/m)\with (1-1/n),\]  then $\{x_n\}_{n\geq1}$  is a Cauchy sequence, hence has a bilimit, say $a$. This means that  for all $x\in X$,  \[\sup_{n\geq1}\inf_{m\geq n}X(x,x_m)=X(x,a), \quad  \sup_{n\geq1}\inf_{m\geq n}X(x_m,x)=X(a,x).  \]
		
		We claim that $a$ is a colimit of $\phi$. To this end it suffices to show that \[\phi =\sup_{n\geq1}\inf_{m\geq n}X(-,x_m).\]  For all $x\in X$ and  $k\geq1$, since \[(1-1/k)\with \sup_{n\geq1}\inf_{m\geq n}X(x,x_m)\leq   \sup_{n\geq k}\inf_{m\geq n}\phi(x_m)\with X(x,x_m)\leq\phi(x),\] it follows that $$\sup_{n\geq1}\inf_{m\geq n}X(x,x_m)\leq\phi(x)$$ by arbitrariness of $k$.     On the other hand, since $$X(x,x_m)\geq \psi(x_m)\with\phi(x)\geq (1-1/m)\with\phi(x)$$  for all $m\geq1$, then \begin{align*} \sup_{n\geq1}\inf_{m\geq n}X(x,x_m)&\geq  \sup_{n\geq1}\inf_{m\geq n}(1-1/m)\with\phi(x)\geq \phi(x). \qedhere \end{align*} \end{proof}

	\begin{definition}(\cite{Lawvere1973}) A real-enriched category $X$ is  Cauchy complete  if  it is separated and every Cauchy weight of $X$ is representable. \end{definition}
	
	For each real-enriched category $X$, the real-enriched category $\CC X$ composed of Cauchy weights of $X$ is Cauchy complete, it is indeed the free Cauchy completion of $X$, see   \cite[Section 7]{Stubbe2005}. The following conclusion was first observed by Lawvere \cite{Lawvere1973} for metric spaces.
	
	\begin{proposition} \label{complete metric}   A  real-enriched category is Cauchy complete if and only if  each of its Cauchy sequences converges uniquely in the open ball topology of its symmetrization. \end{proposition}
	
	The following characterization of Cauchy weights partly motivates the notion of ideals of real-enriched categories in Definition \ref{defn of ideals} below. 
	
	\begin{proposition}\label{Cauchy weight via sub} A weight $\phi$ of a real-enriched category $X$ is Cauchy if and only if the functor $\sub_X(\phi,-)\colon \CP X \to\sV$ preserves all (enriched) colimits, where $\sV=([0,1],\alpha_L)$. \end{proposition}
	
	\begin{proof} If $\psi$ is a left adjoint of $\phi$, then for all $\xi\in\CP X$, $$\sub_X(\phi,\xi)=\xi\swarrow\phi=\xi\circ\psi,$$ from which one deduces that $\sub_X(\phi,-)$ preserves colimits. This proves the necessity. For sufficiency, ssume that the functor  $\sub_X(\phi,-)$ preserves colimits. We wish to show that $\psi\coloneqq X\swarrow\phi$ (or, $\psi(x)= \sub_X(\phi,\sy(x))$ pointwise)  is left adjoint to $\phi$, hence $\phi$ is Cauchy. It suffices to check that $\phi\circ\psi\geq1$. 
	 Since $\sub_X(\phi,-)$ preserves colimits and  $\phi=\sup_{x\in X}(\phi(x)\otimes\sy(x))$,   then \begin{align*}1&=\sub_X(\phi,\phi)=\sup_{x\in X} (\phi(x)\with \sub_X(\phi,  \sy(x)) =\sup_{x\in X}\phi(x)\with\psi(x)=\phi\circ\psi. \qedhere \end{align*}    \end{proof}

	A functor $f\colon X\to Y$  \emph{preserves finite (enriched) colimits} if for any finite real-enriched category $K$ (that means, $K$ has a finite number of objects), any functor $h\colon K\to X$, and any weight $\phi$ of $K$, $f(\colim_\phi h)$ is a colimit of $f\circ h$ weighted by $\phi$ whenever $\colim_\phi h$ exists.

	\begin{definition} \label{defn of ideals} Let $X$ be a real-enriched category. An ideal of $X$ is a weight  $\phi$ of $X$ such that the functor $\sub_X(\phi,-)\colon  \CP X \to \sV$ preserves   finite  colimits, where $\sV=([0,1],\alpha_L)$.  \end{definition}
	
	The term \emph{ideal} is chosen because of the fact that an ideal  of an ordered set $P$ (i.e., a directed lower set of $P$) is exactly a non-empty join-irreducible element in the set of lower sets of $P$ ordered by inclusion.
	
	Every Cauchy weight is an ideal,  every representable weight is an ideal in particular. Since $\CP X$ is complete, it follows from Corollary \ref{calculation of sup by tensors} that  a weight $\phi$  is an ideal if and only if   \begin{enumerate}[label=\rm(I\arabic*)]  
		\item  $\sub_X(\phi,-)\colon  \CP X \to \sV$ preserves tensors; that is, $$\sub_X(\phi,r\with\lam)=r\with\sub_X(\phi,\lam)$$ for all  $\lam\in\CP X$ and all $r\in[0,1]$. \item   $\sub_X(\phi,-)\colon (\CP X)_0\to[0,1]$ preserves finite joins; that is, $$\sub_X(\phi,\lam\vee\mu)=\sub_X(\phi,\lam)\vee\sub_X(\phi,\mu)$$  for all  $\lam,\mu\in\CP X$. \end{enumerate}    
	
	Suppose   $\phi$ is an ideal of $X$. From (I1) it follows that for all $p\in[0,1]$, $$\sub_X(\phi,p_X)=p\with \sub_X(\phi,1_X)=p,$$  which implies  $\phi$ is  \emph{inhabited} in the sense that $\sup_{x\in X}\phi(x)=1$.  
	
	\begin{remark} The notion of ideals of ordered sets has been extended to the quantale-valued context in different ways: weights generated by forward Cauchy nets \cite{AW2011,FS2002,FSW}, flat ideals \cite{Vickers2005,LZZ2020},  irreducible ideals \cite{LZZ2020}, and etc. A comparison of these extension can be found in \cite{LZZ2020}. By   Proposition \ref{characterization of  ideal} below and Theorem 3.13 in \cite{LZZ2020} one sees that for a continuous t-norm,  ideals in the sense of  Definition \ref{defn of ideals} coincide with  irreducible ideals postulated in  \cite[Definition 3.4]{LZZ2020}. \end{remark}
	
	For each real-enriched category $X$, let $$\CI X  $$ be the subcategory of $\CP X$ composed of ideals of $X$. For each functor $f\colon X\to Y$, by the equality  $\sub_Y(f^\ra(\phi),\mu) =\sub_X(\phi,f^\la(\mu))$ for every weight $\phi$ of $X$ and every weight $\mu$ of $Y$, it is readily verified that if $\phi$ is an ideal of $X$, then   $f^\ra(\phi)$ is an ideal of $Y$. Thus, the assignment $X\mapsto\CI X$ defines a  functor $$\CI\colon[0,1]\text{-}\mathbf{Cat}\to[0,1]\text{-}\mathbf{Cat},$$ which is a subfunctor of $\CP\colon[0,1]\text{-}\mathbf{Cat}\to[0,1]\text{-}\mathbf{Cat}$.

	We'll present in Proposition \ref{Yoneda com via ideals}  a characterization for colimits of ideals via Yoneda limits of forward Cauchy nets, parallel to that in Proposition \ref{colimit =bilimit for Cauchy net} for colimits of Cauchy weights (via bilimits of Cauchy nets). For this we need two results  concerning Yoneda limits of forward Cauchy nets. The first deals with Yoneda limits  in the category $\CP X$; the second says that a Yoneda limit of a forward Cauchy net is precisely a colimit of the weight generated by the net. 
	
	\begin{lemma}{\rm (a special case of \cite[Theorem 3.1]{Wagner97})} \label{Yoneda limits in PX} Let $X$ be a real-enriched category. Then, for each forward Cauchy net $\{\phi_i\}_{i\in D}$ of $\CP X$, the weight $\sup_{i\in D}\inf_{j\geq i}\phi_j$ is a Yoneda limit of $\{\phi_i\}_{i\in D}$; that is,  \[\CP X\Big(\sup_{i\in D}\inf_{j\geq i}\phi_j,\phi\Big) =\sup_{i\in D}\inf_{j\geq i}\CP X(\phi_j,\phi) \] for all  $\phi\in\CP X$. \end{lemma}
	
	\begin{lemma} {\rm (\cite[Lemma 46]{FSW})}	\label{yoneda limit as colimits} Let   $\{x_i\}_{i\in D}$ be a forward Cauchy net of a real-enriched category $X$.  Then, an element of $X$ is a Yoneda limit of $\{x_i\}_{i\in D}$ if and only if it is a colimit of the  weight $\sup_{i\in D}\inf_{j\geq i}X(-,x_j)$. \end{lemma}
	
	\begin{proposition}\label{Yoneda com via ideals} Let $X$ be a real-enriched category. Then, every ideal of $X$ has a colimit  if and only if every forward Cauchy net of $X$ has a Yoneda limit. \end{proposition}
	
	\begin{proof} This follows  from a combination of  Lemma \ref{yoneda limit as colimits} and   Proposition \ref{characterization of  ideal} below. \end{proof}
	
	Let $X$ be a real-enriched category. A formal ball of $X$ is a pair $(x,r)$ with $x\in X$ and $r\in[0,1]$, $x$ is called the center and $r$ the radius. For  formal balls  $(x,r)$ and $(y,s)$, define \[(x,r)\sqsubseteq  (y,s) \quad\text{if}~ r\leq s\with X(x,y).\] Then $\sqsubseteq $ is reflexive and transitive, hence an order relation. We write $\mathrm{B}X$ for the set of  formal balls of $X$ ordered by $\sqsubseteq $. The ordered set $\mathrm{B}X$ of formal balls is closely related to $\CP X$.   It is not hard to see that \[(x,r)\sqsubseteq(y,s)\iff 1\leq \CP X(r\with\sy(x),s\with\sy(y)), \] so, when we identify a  formal ball $(x,r)$  with the weight $r\with\sy(x)$,   the order between formal balls coincides with that inherited from $(\CP X)_0$.

	The following proposition is an improvement of related results in  \cite{FS2002,LZZ2020,Vickers2005}. For sake of completeness we include a complete proof here, which is based on the argument of Theorem 3.10 and Theorem 3.13 in \cite{LZZ2020}. It should be noted that  in \cite{LZZ2020} a different order relation on $\mathrm{B}\phi$ is used. 
	
	\begin{proposition}   \label{characterization of  ideal} For each weight $\phi$ of a real-enriched category $X$, the following are equivalent: \begin{enumerate}[label=\rm(\arabic*)]   
			\item  $\phi$ is an   ideal. \item   $\sup_{x\in X}\phi(x)=1$ and  $\mathrm{B}\phi\coloneqq\{(x,r)\in \mathrm{B}X \mid   \phi(x)>r\}$ is a directed subset of $\mathrm{B}X$. \item  $\phi=\sup_{i\in D}\inf_{j\geq i} X(-,x_j)$ for some  forward Cauchy net $\{x_i\}_{i\in D}$ of $X$.\end{enumerate} \end{proposition}
	
	\begin{proof} 
		$(1)\Rightarrow(2)$ We only need to check that $(\mathrm{B}\phi,\sqsubseteq)$ is  directed. Given $(x,r)$ and $(y,s)$ of $\mathrm{B}\phi$, consider the weights $\phi_1$ and $\phi_2$ given by \[\phi_1=  X(x,-)\ra r,\quad  \phi_2 =  X(y,-)\ra s.\] Since $\phi $ is an ideal,  then \begin{align*} \CP X(\phi,\phi_1\vee\phi_2) &=\CP X(\phi,\phi_1)\vee \CP X(\phi,\phi_2) =(\phi(x)\ra r)\vee(\phi(y)\ra s). \end{align*}
		Since \[(\phi(x)\ra r)\vee(\phi(y)\ra s)<1,\]   there exists some $z$ such that \[ \phi(z)>(X(x,z)\ra r)\vee(X(y,z)\ra s) . \] Pick some $t\in[0,1]$ satisfying \[ \phi(z)>t>(X(x,z)\ra r)\vee(X(y,z)\ra s).\] Then    $(z,t)\in \mathrm{B}\phi$. We assert that $(z,t)$ is an upper bound of $(x,r)$ and $ (y,s)$, hence $(\mathrm{B}\phi,\sqsubseteq)$  is directed. To see this, let $u= X(x,z)\ra r$. Then $u$ is, by definition, the largest element of $[0,1]$ with $u\with X(x,z)\leq r$. From $t>u$ we infer that $r< t\with X(x,z)$, hence $(x,r)\sqsubseteq(z,t)$. Likewise, $(y,s)\sqsubseteq(z,t)$.

		$(2)\Rightarrow(3)$ Write an element  in $\mathrm{B}\phi$ as a pair $(x_i,r_i)$ and define a net \[\mathfrak{x}\colon (\mathrm{B}\phi,\sqsubseteq)\to X\] by $\mathfrak{x}(x_i,r_i)=x_i.$ Then $\mathfrak{x}$ is a forward Cauchy net and  for all $x\in X$,  \[ \phi(x)=\sup_{(x_i,r_i)\in\mathrm{B}\phi}\inf_{(x_j,r_j)\sqsupseteq(x_i,r_i) }X(x,x_j). \]

		$(3)\Rightarrow(1)$   Let $\{x_i\}_{i\in D}$ be a forward Cauchy net of $X$. We wish to show that $\phi\coloneqq\sup_{i\in D}\inf_{j\geq i}X(-,x_j)$  is an ideal.

		\textbf{Step 1}. We show that for each weight $\lam$ of $X$, \[\sub_X(\phi,\lam)= \inf_{i\in D}\sup_{j\geq i}\lam(x_j).\]
		
	Since $\lam\colon X^{\rm op}\to \sV$ is a functor,   $\{\lam(x_i)\}_{i\in D}$ is a forward Cauchy net of $\sV^{\rm op}=([0,1], \alpha_R)$, then $\{\lam(x_i)\}_{i\in D}$ converges in the usual topology. Otherwise, it would have two different cluster points,  contradicting that $$\bv_{i\in D}\bw_{i\leq j\leq k}(\lam(x_k)\ra \lam(x_j))=1.$$ So, $\sup_{i\in D}\inf_{j\geq i}\lam(x_j) = \inf_{i\in D}\sup_{j\geq i}\lam(x_j)$. Therefore, \begin{align*}\sub_X(\phi,\lam)&= \sub_X\Big(\sup_{i\in D}\inf_{j\geq i} X(-,x_j),\lam\Big) \\ &
		=\sup_{i\in D}\inf_{j\geq i}\CP X(X(-,x_j),\lam) & \text{(Lemma \ref{Yoneda limits in PX})}\\ &= \sup_{i\in D}\inf_{j\geq i}\lam(x_j)& \text{(Yoneda lemma)}\\ &= \inf_{i\in D}\sup_{j\geq i}\lam(x_j). \end{align*}
	
	\textbf{Step 2}. $\sub_X(\phi,p\with\lam) =p\with\sub_X(\phi,\lam)$ for all  $\lam \in\CP X$ and  $p\in[0,1]$. 
	
	By \textbf{Step 1}, we have \begin{align*}\sub_X(\phi,p\with\lam) &= \inf_{i\in D}\sup_{j\geq i}p\with\lam(x_j) \\ &= p\with \inf_{i\in D}\sup_{j\geq i}\lam(x_j) &\text{($\with$ is continuous)} \\ &= p\with\sub_X(\phi,\lam).\end{align*}
	
	\textbf{Step 3}. $\sub_X(\phi, \lam\vee\mu)=\sub_X(\phi, \lam)\vee\sub_X(\phi, \mu)$ for all  weights  $\lam,\mu$ of $X$.
	
	For this we calculate: 
 \begin{align*} \sub_X(\phi, \lam)\vee\sub_X(\phi, \mu)  &= \inf_{i\in D}\sup_{j\geq i}\lam(x_j)\vee \inf_{i\in D}\sup_{j\geq i}\mu(x_j)& \text{(\textbf{Step 1})} \\ & = \inf_{i\in D}\sup_{j\geq i}(\lam(x_j)\vee\mu(x_j))  \\ & = \sub_X(\phi, \lam\vee\mu).& \text{(\textbf{Step 1})}\end{align*}
		The proof is completed.
	\end{proof}
	
	In category theory, a filtered colimit of representables of a category is called an ind-object of that category, see e.g. \cite[Chapter VI]{Johnstone}. For each forward Cauchy net $\{x_i\}_{i\in D}$ of a real-enriched category $X$, the net $\{\sy(x_i)\}_{i\in D}$ is forward Cauchy in $\CP X$ by the Yoneda lemma, thus,  condition (3) in Proposition \ref{characterization of  ideal} indicates that ideals of real-enriched categories may be viewed  as sort of ind-objects in the real-enriched context. 
	
	A real-enriched category  $X$ is said to be \emph{Yoneda complete}  if  every forward Cauchy net of $X$ has a unique Yoneda limit.  By Proposition \ref{Yoneda com via ideals},    $X$ is Yoneda complete if and only if every ideal of $X$ has a unique colimit. It is trivial that every Yoneda complete real-enriched category is Cauchy complete.
	
	\begin{proposition} \label{IX is Yoneda complete} For each real-enriched category $X$, the category $\CI X$ of ideals of $X$ is Yoneda complete. \end{proposition}
	
	\begin{proof} Actually, $\CI X$ is closed in $\CP X$ under Yoneda limits. To see this, by Proposition \ref{Yoneda com via ideals} and Lemma \ref{yoneda limit as colimits}, it suffices to check that for each ideal $\Phi$ of $\CI X$, the colimit of the inclusion functor $\CI X\to\CP X$ weighted by $\Phi$ is an ideal of $X$, hence a colimit of $\Phi$ in $\CI X$.   Details are left to the reader. \end{proof}

	\begin{lemma}\label{Cauchy net implies Cauchy weight} A forward Cauchy net $\{x_i\}_{i\in D}$ of a real-enriched category $X$ is Cauchy if and only if  the weight $$\phi\coloneqq\sup_{i\in D}\inf_{j\geq i}X(-,x_j)$$ generated by $\{x_i\}_{i\in D}$ is Cauchy. \end{lemma}
	
	\begin{proof}The conclusion is a slight improvement of \cite[Proposition 4.13]{HR2013}. If $\{x_i\}_{i\in D}$ is a  Cauchy net, then $\phi$ is right adjoint to the coweight $$\psi\coloneqq\sup_{i\in D}\inf_{j\geq i}X(x_j,-),$$  hence Cauchy. Conversely, suppose   $\phi$ is a Cauchy weight. By Proposition \ref{adjoint_arrow_calculation}, the left adjoint of $\phi$ is given by $\psi\coloneqq X\swarrow\phi$.   Since,  by Lemma \ref{Yoneda limits in PX},  $\phi$ is a Yoneda limit of the forward Cauchy net $\{X(-,x_i)\}_{i\in D}$ in $\CP X$, it follows that for all $x\in X$, \begin{align*}
			\psi(x) &=\CP X(\phi,X(-,x)) \\
			&=\sup_{i\in D}\inf_{k\geq i}\CP X(X(-,x_k),X(-,x))\\
			&=\sup_{i\in D}\inf_{k\geq i}X(x_k,x).
		\end{align*}
		Therefore,
		\begin{align*}\sup_{i\in D}\inf_{k,j\geq i}X(x_k,x_j)&\geq \sup_{x\in X}\sup_{i\in D}\inf_{k,j\geq i}X(x,x_j)\with X(x_k,x)\\ &\geq \sup_{x\in X}\Big[\Big(\sup_{i\in D}\inf_{j\geq i}X(x,x_j)\Big)\with \Big(\sup_{i\in D}\inf_{k\geq i}X(x_k,x)\Big)\Big]\\ & =\sup_{x\in X} \phi(x)\with \psi(x)\\ & = 1,\end{align*}  showing that $\{x_i\}_{i\in D}$  is a Cauchy net.
	\end{proof}

	Now we present the main result of this section.
	
	\begin{theorem}\label{Smyth via bilimit} A real-enriched category $X$, the following are equivalent: \begin{enumerate}[label=\rm(\arabic*)] 
			\item $X$ is Smyth complete. \item Every forward Cauchy net $\{x_i\}_{i\in D}$ of $X$ has a unique bilimit.  \item $X$ is separated and every ideal of $X$ is representable.
			\item $X$ is Cauchy complete and every ideal of $X$ is a Cauchy weight.  
		\end{enumerate} In this case, $X$ is Yoneda complete. \end{theorem}
	
	\begin{proof} $(1)\Rightarrow(2)$ By Proposition \ref{FC in Smyth is C}, every forward Cauchy net of a Smyth complete real-enriched category is a Cauchy net, then the conclusion follows from Lemma \ref{bilimit of Cauchy net}.

		$(2)\Rightarrow(3)$ It suffices to show that every ideal $\phi$ of $X$ is a Cauchy weight and has a colimit. By Proposition \ref{characterization of ideal}  there is a forward Cauchy net $\{x_i\}_{i\in D}$ of $X$ such that $$\phi=\sup_{i\in D}\inf_{j\geq i} X(-,x_j).$$ By assumption, there is some $a\in X$ such that for all $x\in X$, \[\sup_{i\in D}\inf_{j\geq i}X(x,x_j)=X(x,a), \quad   \sup_{i\in D}\inf_{j\geq i}X(x_j,x)=X(a,x).  \]   
		Putting $x=a$ one sees that $\{x_i\}_{i\in D}$ is a Cauchy net  of $X$, then $\phi$ is a Cauchy weight by Lemma \ref{Cauchy net implies Cauchy weight}, and consequently,  $a$ is a colimit of  $\phi$ by the argument of Proposition \ref{colimit =bilimit for Cauchy net}.

		$(3)\Rightarrow(4)$  Obvious.
		
		$(4)\Rightarrow(1)$ Suppose $\{x_i\}_{i\in D}$ is a forward Cauchy net  of $X$. Since every ideal of $X$ is a Cauchy weight, by Lemma \ref{Cauchy net implies Cauchy weight}   the net $\{x_i\}_{i\in D}$ is Cauchy, so it has a unique bilimit by Cauchy completeness of $X$ and the argument of Proposition \ref{colimit =bilimit for Cauchy net}, then by  Lemma \ref{bilimit of Cauchy net}, it converges uniquely in the open ball topology of the symmetrization of $X$. 
	\end{proof}

	\section{\bf As a real-valued topological property} \label{As a real-valued topological property}
	
	In this section we show that  Smyth completeness can be characterized via properties of the Alexandroff real-valued topology and the Scott real-valued topology of real-enriched categories, without resort to their symmetrization.   
	
	\begin{definition} (\cite{LaiT2017a})
		Let $X$ be a set. A {real-valued topology} on   $X$  is a map $$\delta\colon X\times 2^X\to [0,1]$$ such that for all $x\in X$  and $A,B\in 2^X$,  
		\begin{enumerate}[label=\rm(A\arabic*)] 
			\item $\delta(x,\{x\})=1$;
			\item  $\delta(x,\varnothing)=0$;
			\item  $\delta(x,A\cup B)= \delta(x,A)\vee\delta(x,B)$;
			\item  $\delta(x,A)\geq  \Big(\inf\limits_{b\in B}\delta(b,A)\Big)\with\delta(x,B)$.
		\end{enumerate} The pair $(X,\delta)$ is called a real-valued topological space.  \end{definition} 
	
	The value $\delta(x,A)$ is interpreted as the  degree that $x$ is in the closure of $A$. The axiom (A4) is idempotency of the closure operator in the enriched context. When $\with$ is isomorphic to the product t-norm, a real-valued topological space is essentially an approach space of Lowen \cite{Lowen1997,Lowen2015}.

	\begin{example}(\cite{YuZ2022}) \label{the space K} Define $$\delta_{\mathbb{K}} \colon[0,1]\times2^{[0,1]}\to[0,1]$$ by $$ \delta_{\mathbb{K}}(x,A)= \begin{cases} \inf A \rightarrow x  & A\neq\varnothing,\\ 0  & A=\varnothing \end{cases}$$ for each element $x$ and each subset $A$ of $[0,1]$. Then $\delta_\mathbb{K}$ is a real-valued topology on $[0,1]$. The space $\mathbb{K}=([0,1],\delta_\mathbb{K})$ plays a role in real-valued topology analogous to that of the Sierpi\'{n}ski space in topology.\end{example}   
	
	Suppose   $(X,\delta_X)$ and $(Y,\delta_Y)$ are real-valued topological spaces.  A map $f\colon X\to Y$  is   {continuous}  if 
	$$\delta_X(x,A)\leq\delta_Y(f(x),f(A)) $$ for all $x\in X$ and $A\subseteq X$.  The category of real-valued topological spaces and continuous maps  is denoted by $$[0,1]\text{-}\mathbf{Top}.$$    
	
	As in the case for topological spaces, real-valued topological spaces can be described in different ways, see e.g. \cite{LaiT2017a,Lowen1997,Lowen2015}. In the following we present the  description by closed (fuzzy) sets. In the theory of approach spaces (i.e., real-valued topological spaces with $\with$ being isomorphic to the product t-norm), closed sets are called \emph{lower regular functions} in \cite{Lowen1997,Lowen2015}. We'd like to note that for a general continuous t-norm,  we do not know any ``pleasant''  description of real-valued topological spaces by open (fuzzy) sets. This is one of the difficulties one has to face in real-valued topology.  
	
	Let $(X,\delta)$ be a real-valued topological space. A \emph{closed (fuzzy) set} of  $(X,\delta)$ is  a continuous  function $$\lam\colon (X,\delta)\to ([0,1],\delta_\mathbb{K}),$$    where $([0,1],\delta_\mathbb{K})$ is the space in Example \ref{the space K}.

	\begin{example}\label{cotopology of K} (\cite{YuZ2022}) A   closed set of the real-valued topological space $([0,1], \delta_\mathbb{K})$ is exactly a right continuous coweight of the real-enriched category ${\sf V}=([0,1],\alpha_L)$. Said differently, a closed set of $([0,1], \delta_\mathbb{K})$ is a map $\lambda\colon [0,1]\to[0,1]$ such that \begin{enumerate}[label=(\roman*)]   \item $y\ra x\leq\lambda(y)\ra\lambda(x)$ for all $x,y\in[0,1]$, in particular $\lam$ is monotone;   \item $\lambda(x)=\inf_{y>x}\lam(y)$ for each   $x<1$. \end{enumerate}
	\end{example} 
	
	\begin{proposition} {\rm (\cite{Lowen2015,YuZ2022})} \label{closed sets}
		For each real-valued topological space $(X,\delta)$, the set $\CC_\delta$ of   closed sets  satisfies the following conditions:
		\begin{enumerate}[label=\rm(C\arabic*)]   
			\item  $\lam, \mu\in\CC_\delta\implies \lam\vee\mu\in\CC_\delta$;
			\item {$\lam\in\CC_\delta\implies p\with\lam\in\CC_\delta$} for all $p\in[0,1]$;	\item  $\{\lam_i\}_{i\in I}\subseteq\CC_\delta\implies \inf_{i\in I}\lam_i\in\CC_\delta$; \item {$\lam\in\CC_\delta\implies p\rightarrow\lam\in\CC_\delta$} for all $p\in[0,1]$. \end{enumerate} In other words, $\CC_\delta$ is closed  in $([0,1],\sub_X)$  under formation of finite (enriched) colimits and arbitrary (enriched) limits.  
		
		Conversely, if $\CC\subseteq[0,1]^X$ satisfies (C1)-(C4),   there is a unique real-valued topology $\delta$ on $X$ such that $\CC$ is its set of closed sets. Explicitly, $\delta$ is the coarsest real-valued topology on $X$ that makes $\lam\colon(X,\delta)\to([0,1],\delta_\mathbb{K})$ continuous for all $\lam\in\CC$. 
	\end{proposition}
	
	The category $[0,1]\text{-}\mathbf{Top}$  of real-valued topological spaces contains both the category of topological spaces and the category of real-enriched categories  as full subcategories. In other words, real-valued topological spaces are an common extension of topological spaces and real-enriched categories.
	
	For each topological space $X$, the map $$\delta_X\colon X\times 2^X\to [0,1],\quad
	\delta_X(x,A)=
	\begin{cases}
		1 & \text{if $x$ is in the closure of $A$},\\
		0 & \text{otherwise}
	\end{cases}
	$$
	is  a real-valued topology on $X$. Assigning to each real-enriched category $X$ the real-valued topological space  $\omega(X)\coloneqq(X,\delta_X)$ defines a full and faithful functor
	$$\omega\colon \mathbf{Top}\to [0,1]\text{-}\mathbf{Top}.$$ Spaces of the form $\omega(X)$ are said to be  \emph{topologically generated}. A function $\phi\colon X\to [0,1]$ is a closed set of $\omega(X)$ if and only if it is upper semicontinuous  in the usual sense.
	
	The functor $\omega$ embed $\mathbf{Top}$ in $[0,1]\text{-}\mathbf{Top}$ as a both reflective and coreflective full subcategory. For each real-valued  topological space $(X,\delta)$, a subset $A$ of $X$ is closed in its topological coreflection $\iota(X,\delta)$ if $A=\lam^{-1}(1)=\{x\in X\mid \lam(x)=1\}$ for some closed set $\lam$ of $(X,\delta)$; $A$ is closed in its topological reflection $\varrho(X,\delta)$ if the characteristic map $1_A$ is a closed set of $(X,\delta)$.

	For a real-enriched category $(X,\alpha)$, define  $\Gamma(\alpha)\colon X\times 2^X\to[0,1] $ by $$
	\Gamma(\alpha)(x,A)=
	\begin{cases}
		0 & A=\varnothing,\\
		\sup\limits_{y\in A}\alpha(x,y) & A\neq\varnothing.
	\end{cases}
	$$ Then $\Gamma(\alpha)$ is a real-valued topology on $X$. 
	Assigning $(X,\Gamma(\alpha))$ to $(X,\alpha)$ gives rise to a full and faithful functor
	$$\Gamma\colon [0,1]\text{-}\mathbf{Cat}\to [0,1]\text{-}\mathbf{Top}. $$ 
	Spaces of the form $\Gamma(X,\alpha)$ are called \emph{Alexandroff real-valued topological spaces}. A function $\phi\colon X\to [0,1]$ is a closed set of $\Gamma(X,\alpha)$ if and only if $\phi$ is a weight of $(X,\alpha)$.  
	
	The functor $\Gamma$ has a right adjoint   $$\Omega\colon [0,1]\text{-}\mathbf{Top}\to [0,1]\text{-}\mathbf{Cat}$$ which maps   $(X,\delta)$ to    $\Omega(X,\delta)=(X,\Omega(\delta)),$ where  $\Omega(\delta)(x,y)=\delta(x,\{y\}).$ 
	The real-enriched category  $\Omega(X,\delta)$  is called  the \emph{specialization}   of  $(X,\delta)$,  it is obtained by specializing the second argument in $\delta(x,A)$ to singleton sets.

	\begin{proposition}For each real-enriched category $X$, the  topological coreflection of the Alexandroff real-valued topology of $X$ coincides with the open ball topology of $X$. \end{proposition} 
	
	\begin{proof}  Assume that $U$ is an open set of $\iota\circ\Gamma (X)$. Then there is a functor $\lam\colon X^{\rm op}\to{\sf V}$ such that $U$ is the complement of $\{z\in X\mid \lam(z)=1\}$. If $x\in U$, then $\lam(x)<1$. Pick some $r<1$ such that $r\ra\lam(x)<1$. Then for all $y\in B(x,r)$, $\lam(y)\leq X(x,y)\ra \lam(x)<1$, which implies that $U$ is open in the open ball topology. Conversely, we show that every open ball $B(x,r)$ is an open set of $\iota\circ\Gamma (X)$. Consider the weight $\phi=X(x,-)\ra r$. It is clear that $\phi(z)=1$ if and only if $z\notin B(x,r)$, hence $B(x,r)$ is open in $\iota\circ\Gamma (X)$. \end{proof}

	Suppose $X$ is a real-enriched category.  By a \emph{Scott closed (fuzzy) set}  of $X$ we mean a functor $\lam\colon X^{\rm op}\to{\sf V}$ (also viewed as a weight of $X$) such that for every ideal $\phi$ of $X$, $$\sub_X(\phi,\lam)\leq \lam(\colim\phi)$$   whenever  $\phi$ has a colimit. The inequality is actually an equality.  It  says, semantically, that if $\phi$ is contained in $\lam$, then its colimit belongs to $\lam$. 
	
	Suppose $\lam\colon X^{\rm op}\to{\sf V}$ is a functor.
	Since for each weight $\phi$ of $X$,  the colimit of the functor $\lam^{\rm op}\colon X\to{\sf V}^{\rm op}$ weighted by $\phi$  is  the limit of $\lam\colon X^{\rm op}\to{\sf V}$ coweighted by the coweight $\phi^{\rm op}$ of $X^{\rm op}$, from Example \ref{sub as limit} it follows that ${\colim}_\phi\lam^{\rm op} =\lim_{\phi^{\rm op}}\lam=\sub_X(\phi,\lam).$ So, a functor $\lam\colon X^{\rm op}\to{\sf V}$  is Scott closed if and only if its opposite $\lam^{\rm op}\colon X\to{\sf V}^{\rm op}$ is Yoneda continuous.
	
	The Scott closed sets of a real-enriched category $X$ satisfy (C1)-(C4) in Proposition \ref{closed sets}, hence determine a real-valued topology, called the \emph{Scott real-valued topology}, on $X$. The resulting real-valued topological space is denoted by $\Sigma(X)$. The  Scott real-valued topology of real-enriched categories is an extension of the Scott approach structure for quasi-metric spaces in \cite{LiZ18b,Windels}.

	\begin{example}Since a functor $\phi\colon  {\sf V}^{\rm op}\to {\sf V}^{\rm op}$   is Yoneda  continuous if and only if it is right continuous, it follows from Example \ref{cotopology of K} that   $\Sigma({\sf V}^{\rm op})$ is the space $([0,1], \delta_\mathbb{K})$ in Example \ref{the space K}.  
	\end{example}
	
	Since every functor preserves  colimits of Cauchy weights, then $\Gamma(X)=\Sigma(X)$ for each Smyth complete real-enriched category $X$. The converse is also true for Yoneda complete real-enriched categories.
	
	\begin{theorem}\label{Smyth via gamma=sigma} A Yoneda complete real-enriched category $X$ is Smyth complete if and only if $\Gamma(X)=\Sigma(X)$. \end{theorem}
	
	\begin{proof}We only need to check the sufficiency. Suppose on the contrary that $X$ is not Smyth complete, then $X$ has an ideal, say $\phi$, that is not a Cauchy weight, hence not representable  and consequently,  $\phi(\colim\phi)<1$. Since $\sub_X(\phi,\phi)=1$, it follows that $\phi$ is not Scott closed, contradicting that  $\Gamma(X)=\Sigma(X)$. \end{proof}
	
	Yoneda completeness is indispensable in the above proposition.  In the following we  characterize Smyth completeness via  sobriety of the Alexandroff real-valued topology. 
	
	Given a real-valued topological space $(X,\delta)$, we say that a closed set $\phi$ of $(X,\delta)$  is \emph{irreducible} if  the functor $$\sub_X(\phi,-)\colon (\CC_\delta,\sub_X)\to([0,1],\alpha_L)$$ preserves finite  enriched  colimits.
	
	Since  $\CC_\delta$ is closed  in $([0,1],\sub_X)$  under formation of finite (enriched) colimits and arbitrary (enriched) limits, then a closed set  $\phi$ is irreducible if for all  $\lam,\mu \in\CC_\delta$ and   $r\in[0,1]$, \begin{enumerate}[label=\rm(\roman*)]   
		\item $\sub_X(\phi,r\with\lam) =r\with\sub_X(\phi,\lam)$, 
		\item $\sub_X(\phi,\lam\vee\mu)= \sub_X(\phi,\lam)\vee\sub_X(\phi,\mu)$.  \end{enumerate} 
	
	Every irreducible closed $\phi$ is clearly inhabited, i.e., $\bv_{x\in X}\phi(x)=1$.
	
	\begin{definition}  
		A real-valued topological space $(X,\delta)$  is sober if for each   irreducible closed set $\phi$, there is a unique element $x$ of $X$ such that $\phi=\delta(-,\{x\})$. \end{definition} 
	
	The above postulation is a special case of that for sober fuzzy cotopological spaces in \cite{Zhang2018}. When the continuous t-norm is isomorphic to the product t-norm, it  reduces to that of sober approach spaces in \cite{BRC}. 
	
	It is not hard to check that for each topological space $X$, the real-valued topological space $\omega(X)$ is sober if, and only if, $X$ is sober as a topological space. So, the notion of sobriety of real-valued topological spaces extends that of topological spaces.
	
	\begin{theorem}A   real-enriched category  $X$ is Smyth complete if and only if the  real-valued topological space $\Gamma(X)$ is sober.  \end{theorem}
	
	\begin{proof} The conclusion was first proved in \cite{LiZ18a} for quasi-metric spaces. It follows from the fact that the ideals of  $X$ are precisely the irreducible closed sets of the real-valued topological space $\Gamma(X)$. \end{proof}
	
	\section{\bf Smyth completable real-enriched categories} 
	\label{Smyth completable}
	
	A separated real-enriched category $X$ is \emph{Smyth completable} if there is a Smyth complete real-enriched category and a fully faithful functor $f\colon X\to Y$.

	\begin{theorem}\label{characterizing Smyth completable} For a separated real-enriched category $X$, the following are equivalent:  \begin{enumerate}[label={\rm(\arabic*)}]   
			\item $X$ is Smyth completable.  \item Every ideal of $X$ is a Cauchy weight. \item Every forward Cauchy net of $X$ is a Cauchy net.  \item Every ideal of $\CI X$ is representable, hence $\CI(\CI X)=\CI X$.  
	\end{enumerate} \end{theorem}
	To prove the theorem we need a lemma which says that each real-enriched category is isomorphic to its Cauchy completion in the category of distributors. 
	\begin{lemma} \label{X is isomorphic to CX} {\rm (\cite[Proposition 7.11]{Stubbe2005})} Suppose  $X$ is a real-enriched category and   $\sy\colon X\to\CC X$ is the Yoneda embedding with codomain restricted to the set $\CC X$ of Cauchy weights. Then, $\sy_*\circ\sy^*=\CC X$ and $\sy^*\circ \sy_*=X$. Therefore, $X$ and $\CC X$ are isomorphic in the category of distributors. 
	\end{lemma}    
	
	
	\begin{proof}[Proof of Theorem \ref{characterizing Smyth completable}] The equivalence $(2)\Leftrightarrow(3)$ follows from Proposition \ref{characterization of ideal} and Lemma \ref{Cauchy net implies Cauchy weight}.
		
		$(1)\Rightarrow(2)$ Suppose  $f\colon X\to Y$ is fully faithful with $Y$ being Smyth complete. Let $\phi$ be an ideal of $X$. Then $\phi\circ f^*$ is an ideal, hence a Cauchy weight of $Y$. Let $\psi\colon\star\oto Y$ be a left adjoint of $\phi\circ f^*\colon Y\oto\star$. We  show that $f^*\circ\psi\colon \star\oto X$ is a left adjoint of $\phi\colon X\oto\star$. First, since $\psi\dashv \phi\circ f^*$, then $\phi\circ(f^*\circ\psi) =(\phi\circ f^*)\circ\psi\geq1$. Next, \begin{align*}\psi\dashv \phi\circ f^* &\implies \psi\circ\phi\circ f^*\leq Y \\ &\implies \psi\circ\phi\leq f_* &(f^*\circ f_*=X) \\ &\implies (f^*\circ\psi)\circ\phi\leq X. &(f^*\circ f_*=X)\end{align*} Therefore, $f^*\circ\psi$ is a left adjoint of $\phi$.
		
		$(2)\Rightarrow(4)$ Since $\CI X=\CC X$ by assumption,  we only need to check that every ideal of $\CC X$ is representable. Suppose that $\phi\colon\CC X\oto\star$ is an ideal of $\CC X$. By Lemma \ref{X is isomorphic to CX}, $\sy_*\colon X\oto\CC X$ is an isomorphism in the category of distributors, it follows that  $\phi\circ \sy_*\colon X\oto\CC X\oto\star$ is an ideal, hence a Cauchy weight of $X$, then $\phi$ is a Cauchy weight of $\CC X$ because $\sy_*$ is an isomorphism. Therefore, $\phi$ is representable since $\CC X$ is Cauchy complete.

		$(4)\Rightarrow(1)$ Since every ideal of $\CI X$ is representable, then $\CI X$ is Smyth complete and consequently, $X$ is Smyth completable. \end{proof}
	
	\begin{corollary}A separated real-enriched category  is Smyth completable if and only if its Cauchy completion is Smyth complete. \end{corollary}
	\begin{proof} Sufficiency is clear, for necessity assume that $f\colon X\to Y$ is a fully faithful functor with $Y$ Smyth complete. Since $Y$ is Cauchy complete, there is a  unique functor  $\overline{f}\colon\CC X\to Y$ that extends $f$. It is not hard to check that $\overline{f}$ is also fully faithful, hence $\CC X$ is Smyth completable. By Theorem \ref{characterizing Smyth completable} every ideal of $\CC X$ is Cauchy, hence representable by Cauchy completeness of $\CC X$, and consequently, $\CC X$ is Smyth complete. \end{proof}
	
	In the following we show that every Smyth completable real-enriched category is continuous (to be defined below). 
	For each real-enriched category $X$, let $$\CJ X=\{\phi\in\CI X\mid \phi ~\text{has a colimit}\}.$$ It is clear that every forward Cauchy net of $X$ has a Yoneda limit if and only if $\CJ X=\CI X$.
	
	The Yoneda embedding $\sy\colon X\to\CP X$ factors through $\CJ X$, so we'll also use the symbol $\sy$ to denote the functor $X\to\CJ X$ obtained by restricting the codomain of the Yoneda embedding to $\CJ X$. It is clear that $\sy\colon X\to\CJ X$ has a left adjoint $\colim\colon \CJ X\to X$ that sends each $\phi\in\CJ X$ to its colimit.
	
	\begin{definition}Suppose $X$ is a real-enriched category. If   $\colim\colon \CJ X\to X$ has a left adjoint, then we say that $X$  is continuous. In other words, $X$  is continuous if there exists a string of adjunctions  \[\wayb\dashv\colim \dashv \sy\colon X\to\CJ X.\]    A Yoneda comlete and continuous real-enriched category is called a real-enriched domain. \end{definition}
	
	\begin{example}The real-enriched category $\sV^{\rm op}$ is always a real-enriched domain  \cite[Corollary 4.13]{LaiZ2020};  the real-enriched category $\sV$ is a real-enriched domain if and only if the implication operator $\ra\colon[0,1] \times[0,1]\to[0,1]$ is continuous at each point off the diagonal  \cite[Theorem 6.4]{LaiZ2020}. \end{example}
	
	\begin{proposition}A real-enriched category $X$ is a real-enriched domain if and only if it is continuous and the  real-valued topological space $\Sigma(X)$ is sober.  \end{proposition}
	
	\begin{proof}
		Sufficiency follows from the fact that the specialization of a sober real-valued topological space $(X,\delta)$ is Yoneda complete \cite[Proposition 3.10]{Zhang2018}. Necessity follows from the fact that the Scott real-valued topology of each real-enriched domain is sober \cite[Proposition 6.9]{YuZ2022}.\end{proof}
	
	The \emph{way below distributor}  on a real-enriched category $X$ is  the distributor $\mathfrak{w}\colon X\oto X$ given by \[\mathfrak{w}(y,x)=\inf_{\phi\in\CJ X}(X(x,\colim\phi)\ra\phi(y))\] for all $x,y\in X$.  It is obvious that $\mathfrak{w}(y,x)\leq X(y,x)$.
	
	\begin{proposition}\label{way below as left adjoint} A real-enriched category  $X$ is continuous   if and only if for all $x\in X$, the weight $\mathfrak{w} (-,x)$ belongs to $\CJ X$ with $x$ being a  colimit. In this case,    $\wayb x=\mathfrak{w} (-,x)$. 
	\end{proposition}
	
	\begin{proof} We prove the necessity first. Suppose $X$ is continuous. We show that  for each $x\in X$, $\mathfrak{w} (-,x)=\wayb x$,   hence $\mathfrak{w} (-,x)$ belongs to $\CJ X$ and has $x$ as  colimit.   Since $\colim\wayb x=x$, then $$\mathfrak{w} (-,x)\leq X(x,\colim\wayb x)\ra\wayb x= \wayb x.$$ On the other hand, since  for all   $\phi\in\CJ X$, we have $$X(x,\colim\phi)\with\wayb x = \CJ X(\wayb x,\phi)\with\wayb x\leq \phi,$$  then $\wayb x\leq X(x,\colim\phi)\ra\phi$ for all $\phi\in\CJ X$, hence $\wayb x\leq \mathfrak{w} (-,x)$.  
		
		For sufficiency we show that the assignment $x\mapsto \wayb x\coloneqq \mathfrak{w} (-,x)$ defines a  left  adjoint of $\colim\colon\CJ X\to X$. Let $x\in X$ and $\phi\in\CJ X$. By definition   we have $\wayb x\leq   X(x,\colim\phi)\ra\phi$, hence $$X(x,\colim\phi)\leq \CJ X(\wayb x,\phi).$$  Since taking colimit is functorial $\CJ X\to X$ and $\colim\wayb x=x$, then $$\CJ X(\wayb x,\phi)\leq X(\colim\wayb x,\colim\phi)=X(x,\colim\phi).$$ Therefore, $\wayb$ is a left adjoint of $\colim\colon\CJ X\to X$ and consequently, $X$ is continuous. \end{proof} 
	
	\begin{proposition}The way-below distributor $\mathfrak{w}$ on a continuous real-enriched category   interpolates in the sense that $\mathfrak{w} \circ \mathfrak{w} =\mathfrak{w} $. \end{proposition}
	\begin{proof}  By Proposition \ref{way below as left adjoint}, $\mathfrak{w}(-,x)$ is the least element of $\CJ X$ having $x$ as colimit. Since $\mathfrak{w}(y,x)\leq X(y,x)$ for all $y,x\in X$, it follows that $\mathfrak{w} \circ \mathfrak{w}(-,x)\leq\mathfrak{w}(-,x)$, so  we  only need to show that $\mathfrak{w}\circ\mathfrak{w}(-,x)$ belongs to $\CI X$ and has $x$ as colimit. 
		
		\textbf{Step 1}. $\mathfrak{w}\circ\mathfrak{w}(-,x)\in \CI X$. 
		
		Let $\lam,\mu\in\CP X$. Since $$\sub_X(\mathfrak{w}(-,y),\lam)\with X(z,y)\with \mathfrak{w}(-,z)
		\leq \lam, $$ it follows that $$\sub_X(\mathfrak{w}(-,y),\lam)\with X(z,y)\leq \sub_X(\mathfrak{w}(-,z),\lam), $$ so the assignment $y\mapsto \sub_X(\mathfrak{w}(-,y),\lam)$ defines a weight of $X$. Then,  \begin{align*} &\quad\quad ~  \sub_X(\mathfrak{w}\circ\mathfrak{w}(-,x),\lam\vee\mu)\\ & =\bw_{y\in X}\sub_X(\mathfrak{w}(y,x)\with\mathfrak{w}(-,y),\lam\vee \mu)\\ &= \bw_{y\in X}\Big(\mathfrak{w}(y,x)\ra\sub_X(\mathfrak{w}(-,y),\lam\vee \mu)\Big)\\ &=\bw_{y\in X}\Big(\mathfrak{w}(y,x)\ra(\sub_X(\mathfrak{w}(-,y),\lam)\vee\sub_X(\mathfrak{w}(-,y),\mu)\Big)\\ &=\bw_{y\in X}\Big(\mathfrak{w}(y,x)\ra\sub_X(\mathfrak{w}(-,y),\lam)\Big)\vee\bw_{y\in X}\Big(\mathfrak{w}(y,x)\ra\sub_X(\mathfrak{w}(-,y),\mu)\Big)\\ &=\sub_X(\mathfrak{w}\circ\mathfrak{w}(-,x),\lam)\vee\sub_X(\mathfrak{w}\circ\mathfrak{w}(-,x),\mu).\end{align*} Likewise, one verifies that for all  $\lam\in\CP X$ and  $r\in[0,1]$, $$\sub_X(\mathfrak{w}\circ\mathfrak{w}(-,x),r\with\lam)=r\with\sub_X(\mathfrak{w}\circ\mathfrak{w}(-,x),\lam).$$  This proves that $\mathfrak{w}\circ\mathfrak{w}(-,x)\in \CI X$. 
		
		\textbf{Step 2}. $\colim\mathfrak{w}\circ\mathfrak{w}(-,x)=x$.
		
		For all $b\in X$, \begin{align*} X(x,b) &=  \bw_{y\in X}(\mathfrak{w}(y,x)\ra X(y,b)) &(\colim\mathfrak{w}(-,x)=x) \\ &= \bw_{y\in X}(\mathfrak{w}(y,x)\ra\CP X(\mathfrak{w}(-,y),X(-,b))) &(\colim\mathfrak{w}(-,y)=y)\\ &= \CP X(\colim\mathfrak{w}\circ\mathfrak{w}(-,x),X(-,b)) , \end{align*} hence $\colim\mathfrak{w}\circ\mathfrak{w}(-,x)=x$.
	\end{proof}

	The following definition is a direct extension of that for generalized metric spaces in \cite{BBR98,Goubault}.
	
	\begin{definition}
		An element $a$ of a real-enriched category $X$  is compact if the functor $$X(a,-)\colon X\to\sV$$  is Yoneda continuous, where $\sV=([0,1],\alpha_L)$.  \end{definition}
	
	\begin{lemma}\label{characterizing compact element} For each element $a$ of a real-enriched category $X$, the following  are equivalent: \begin{enumerate}[label=\rm(\arabic*)] \item    $a$   is  compact. \item $X(a,\colim\phi) =\phi(a)$ for all  $\phi\in\CJ X$.  \item $\mathfrak{w}(-,a)=X(-,a)$. \end{enumerate} \end{lemma}
	
	\begin{proof} For all $\phi\in \CJ X$, by Example \ref{cpt as colimit} the colimit of   $X(a,-)\colon X\to\sV$ weighted by $\phi$ is equal to the composite of $X(a,-)$ and $\phi$ as distributors; that is,   $\colim_\phi  X(a,-) = \phi\circ X(a,-)=\phi(a)$.  This proves $(1)\Leftrightarrow(2)$.  
		
		$(2)\Rightarrow(3)$ For each $y\in X$, \begin{align*}\mathfrak{w}(y,a)&=\inf_{\phi\in\CJ X}(X(a,\colim\phi)\ra\phi(y))\\ & = \inf_{\phi\in\CJ X}(\phi(a)\ra\phi(y))\\ & = X(y,a).\end{align*}
		
		$(3)\Rightarrow(2)$ Since taking colimit is functorial, then $$\phi(a)=\CP X(X(-,a),\phi)\leq X(a,\colim\phi) $$ for all $\phi\in\CJ X$. On the other hand, since, by definition,  $$\mathfrak{w}(y,a)=\inf_{\phi\in\CJ X}(X(a,\colim\phi)\ra\phi(y)),$$  then $X(a,\colim\phi)\leq \mathfrak{w}(y,a)\ra \phi(y)$ for all $y\in X$ and all $\phi\in\CJ X$,  hence $$X(a,\colim\phi)\leq \CP X(\mathfrak{w}(-,a),\phi)=\CP X(X(-,a),\phi)=\phi(a).$$ Therefore, $X(a,\colim\phi)=\phi(a)$. \end{proof}
	
	A real-enriched category $X$ is  \emph{algebraic} if it is Yoneda complete and  each of its elements is a Yoneda limit of a forward Cauchy net consisting of compact elements of $X$. Said differently, $X$  is algebraic if it is Yoneda complete and for each $x\in X$, there is an ideal $\phi$ of the real-enriched category $K(X)$ composed of compact elements of $X$ such that $x$ is a colimit of the inclusion functor $i\colon K(X)\to X$ weighted by $\phi$.  
	
	\begin{example}\label{compact element of V}  Consider the real-enriched category $\sV=([0,1],\alpha_L)$.  For each $a\in[0,1]$, let $a^+$ be the least idempotent element in $[a,1]$. If $0<a\leq a^+<1$, consider the ideal $\phi=\bv_{x<a}\sV(-,x)$. Since $\colim\phi =a$ and  $$\phi(a)=\bv_{x<a}(a\ra x)\leq a^+<1=\sV(a,\colim\phi),$$ if follows that $a$ is not compact. By this fact one readily verifies that $\sV$ is algebraic if and only if the continuous t-norm $\with$ is Archimedean. Likewise, $\sV^{\rm op}$ is algebraic if and only if $\with$ is Archimedean. \end{example}
	
	\begin{proposition}\label{algebraic is continuous}  Every  algebraic real-enriched category is a real-enriched domain.   \end{proposition}
	
	\begin{proof}  
		It suffices to check that the functor $\colim\colon \CI X  \to X$   has a left adjoint. For each  $a\in X$, take a forward Cauchy net $\{a_i\}_{i\in D}$ of compact elements of $X$ with $a$ as a Yoneda limit, let $\wayb a$ be the Yoneda limit of the forward Cauchy net $\{\sy(a_i)\}_{i\in D}$ in $\CI X$ (existence is guaranteed by Proposition \ref{IX is Yoneda complete}). Then, for every ideal $\phi$ of $X$, \begin{align*} \CI X(\wayb a,\phi) &=\inf_{i\in D}\sup_{j\geq i}\CI X(\sy(a_j), \phi)&(\text{$\wayb  a$ is a Yoneda limit})\\ &=\inf_{i\in D}\sup_{j\geq i}\phi(a_j)&({\rm Yoneda~lemma}) \\ &=\inf_{i\in D}\sup_{j\geq i}X(a_j, \colim\phi)&( a_j{\rm~is~compact}) \\ &=X(a,\colim\phi),&(a~ \text{is a Yoneda limit}) \end{align*} hence $\wayb$ is a left adjoint of $\colim$. \end{proof}
	
	\begin{proposition}For a Smyth completable real-enriched category $X$, the functor $\sy\colon X\to\CJ X$ is both left and right adjoint to the functor $\colim\colon \CJ X\to X$. In particular, $X$ is continuous. \end{proposition}
	
	\begin{proof} First we show that every element of $X$ is compact; that means, $X(a,\colim\phi)=\phi(a)$ for all $a\in X$ and all $\phi\in \CJ X$. Since $X$ is Smyth completable, $\phi$ is   Cauchy, hence representable because it has a colimit. This means $\phi=X(-,b)$ for some $b\in X$. Therefore, \begin{align*} X(a,\colim\phi)   &= X(\colim X(-,a),b) \\ &=\CP X(X(-,a), X(-,b))\\ &=\CP X(X(-,a), \phi)\\ &=\phi(a).  \end{align*}
		
		Next we show that   $\sy\colon X\to\CJ X$   is left adjoint to $\colim\colon \CJ X\to X$. This is easy, since $ \CJ X(\sy(a),\phi)=\phi(a)= X(a,\colim\phi)$ for all $a\in X$ and $\phi\in \CJ X$. \end{proof}
	
	
	\begin{corollary}Every Smyth complete real-enriched category is algebraic, hence a real-enriched domain. \end{corollary}
	
	\begin{corollary}For a Yoneda complete real-enriched category $X$, the following are equivalent:  \begin{enumerate}[label={\rm(\arabic*)}]  
			\item $X$ is Smyth complete.  \item  The functor $\sy\colon X\to\CI X$  is both  left   and  right adjoint to  the functor $\colim\colon \CI X\to X$. \item  Every element of $X$ is compact.  
	\end{enumerate}\end{corollary}
	
	\bibliographystyle{plain}
	
\end{document}